\documentclass[11pt]{amsart}

\usepackage[latin1]{inputenc}
\usepackage{amsmath,amsfonts,amssymb}
\usepackage{amscd}
\usepackage{graphicx}

 \normalbaselines
\newcommand{\rad}{\operatorname{rad}}

\newcommand{\soc}{\operatorname{soc}}

\newtheorem{theorem}{Theorem}[section]
\newtheorem{prop}[theorem]{Proposition}

\newtheorem{lemma}[theorem]{Lemma}

\newtheorem{remark}[theorem]{Remark}

\begin{document}

\title[Generalized Reynolds ideals]
{Generalized Reynolds ideals and derived equivalences for
algebras of dihedral and semidihedral type}

\author{Thorsten Holm \and Alexander Zimmermann}

\address{~~\newline
Thorsten Holm\newline
Leibniz Universit\"at Hannover, 
Institut f\"ur Algebra, Zahlentheorie und Diskrete Mathematik,
Welfengarten 1, 30167 Hannover, Germany}
\email{holm@math.uni-hannover.de}
\urladdr{http://www.iazd.uni-hannover.de/\~{ }tholm}

\address{
~~\newline Alexander Zimmermann\newline  Universit\'{e} de
Picardie, Facult\'e de Math\'ematiques et LAMFA (UMR 6140 du
CNRS), 33 rue St Leu, 80039 Amiens CEDEX 1, France}
\email{alexander.zimmermann@u-picardie.fr}
\urladdr{http://www.lamfa.u-picardie.fr/alex/azimengl.html}

\thanks{The authors are supported by a German-French grant DAAD-PROCOPE,
DAAD grant number D/0628179, and
``partenariat Hubert Curien PROCOPE dossier 14765WB'' respectively. We
gratefully acknowledge the financial support which made the present work
possible.}



\bigskip

\begin{abstract}
Generalized Reynolds ideals are ideals of the center
of a symmetric algebra over a field of positive characteristic.
They have been shown by the second author to be invariant
under derived equivalences.
In this paper we determine the generalized Reynolds ideals of
algebras of dihedral and semidihedral type (as defined by Erdmann),
in characteristic 2. In this way we solve
some open problems about scalars occurring in the
derived equivalence classification of
these algebras.
\end{abstract}

\maketitle


\section{Introduction}

Finite-dimensional algebras are distinguished according to their
representation type, which is either finite, tame or wild. For
blocks of group algebras the representation type is determined by
the structure of the defect group. It is finite if and only if the
defect groups are cyclic. The structure of such blocks is known;
in particular these algebras are Brauer tree algebras. Blocks of
tame representation type occur only in characteristic 2, and then
precisely if the defect groups are dihedral, semidihedral or
generalized quaternion. The structure of such blocks has been
determined in a series of seminal papers by K. Erdmann
\cite{Erdmann}. She introduced the more general classes of
algebras of dihedral, semidihedral and quaternion type, and
classified them by explicitly describing their basic algebras by
quivers and relations. However, some subtle questions remained
open in her classification, most of them related to scalars
occurring in the relations. Based on Erdmann's Morita equivalence
classification, algebras of dihedral, semidihedral and quaternion
type have been classified up to derived equivalence by the first
author in \cite{H-tame}, \cite{H-derived}. Along the way, some of
the subtle remaining problems in \cite{Erdmann} have been solved,
but not all. In particular, for the case of two simple modules
still scalars occur in the relations, and it could not be decided
whether the algebras for different scalars are derived equivalent,
or not. (See the appendix of \cite{H-habil} for tables showing
the status of the derived equivalence classifications.)

In this paper, we shall study new invariants for symmetric algebras
$A$ over fields of positive characteristic
which have been defined in \cite{HHKM}. These are descending
sequences of so-called generalized Reynolds ideals, of the center,
$$Z(A)\supseteq T_1(A)^{\perp}\supseteq T_2(A)^{\perp}\supseteq
\ldots \supseteq T_n(A)^{\perp}\supseteq\ldots .$$
The precise definition of these ideals
is given in Section \ref{Sec-Reynolds}
below.

It has been
shown by the second author in \cite{Z} that these sequences of ideals
are invariant under derived equivalences, i.e. any derived
equivalence implies an isomorphism between the centers mapping
the generalized Reynolds ideals onto each other.

It turns out that generalized Reynolds ideals can be very useful
for distinguishing algebras up to derived equivalence. For
instance, in \cite{HS}, generalized Reynolds ideals have been used
successfully to complete
the derived equivalence classification of symmetric algebras
of domestic representation type.

In this paper,
we are going to compute the generalized Reynolds ideals for algebras
of dihedral and semidihedral type. As main application we will
settle some of the
scalar problems which remained open in the derived equivalence
classification \cite{H-tame}, \cite{H-derived}.

Using the notation of \cite{Erdmann}, our results can be summarized
as follows. The definitions of the algebras under consideration
are also recalled below in Sections \ref{sec-dihedral} and
\ref{sec-semidihedral}, respectively.

\begin{theorem}
\label{thm-intro-dihedral}
Let $F$ be an algebraically closed field of characteristic 2.
For any given integers $k,s\ge 1$
consider the algebras of dihedral type
$D(2A)^{k,s}(c)$ for the scalars $c=0$ and $c=1$.
Suppose that if $k=2$ then $s\ge 3$ is odd, and if $s=2$ then
$k\ge 3$ is odd.

Then the algebras $D(2A)^{k,s}(0)$ and $D(2A)^{k,s}(1)$
have different sequences of generalized
Rey\-nolds ideals. In particular, the algebras $D(2A)^{k,s}(0)$ and $D(2A)^{k,s}(1)$
are not derived
equivalent.
\end{theorem}

The above result has also been obtained earlier by M. Kauer
\cite{Kauer}, \cite{KR}, using entirely different methods.
However, our new proof seems to be more elementary, just
using linear algebra calculations.

For algebras of semidihedral type, we can prove the following result using generalized
Reynolds ideals.

\begin{theorem}
\label{thm1-intro-semidihedral}
Let $F$ be an algebraically closed field of characteristic 2.
For any given integers $k,s\ge 1$, consider the algebras of
semidihedral type
$SD(2B)_1^{k,s}(c)$ for the scalars $c=0$ and $c=1$.
Suppose that if $k=2$ then $s\ge 3$ is odd, and if $s=2$ then
$k\ge 3$ is odd.

Then the algebras $SD(2B)_1^{k,s}(0)$ and $SD(2B)_1^{k,s}(1)$
have different sequences of generalized
Reynolds ideals. In particular, the algebras $SD(2B)_1^{k,s}(0)$ and
$SD(2B)_1^{k,s}(1)$ are not derived equivalent.
\end{theorem}

This settles an important open problem in the derived equivalence
classification of algebras of semidihedral type.
However, this does not yet complete this classification; there
is a second family $SD(2B)_2^{k,s}(c)$ involved for which we can
prove the following partial result.

\begin{theorem}
\label{thm2-intro-semidihedral}
Let $F$ be an algebraically closed field of characteristic 2.
For any given integers $k,s\ge 1$, consider the algebras of
semidihedral type
$SD(2B)_2^{k,s}(c)$ for the scalars $c=0$ and $c=1$.
If the parameters $k$ and $s$ are both odd, then
the algebras $SD(2B)_2^{k,s}(0)$ and $SD(2B)_2^{k,s}(1)$
have different sequences of generalized
Reynolds ideals. In particular, for $k$ and $s$ odd,
the algebras $SD(2B)_2^{k,s}(0)$ and
$SD(2B)_2^{k,s}(1)$ are not derived equivalent.
\end{theorem}

Here in the semidihedral case, in
order to distinguish derived equivalence classes in the
remaining cases new derived
invariants would have to be discovered.

\section{Generalized Reynolds ideals}
\label{Sec-Reynolds}

The aim of this section is to briefly give the necessary
background on generalized Reynolds ideals,
as introduced by B. K\"ulshammer
\cite{Ku1}. For more details we
refer to the survey \cite{Ku2}.
For recent developments
we also refer to \cite{BHHKM}, \cite{HHKM},
\cite{Z}, \cite{Z2}.

Let $F$ an algebraically closed field of characteristic $p > 0$.
(For the theory of generalized Reynolds ideals a perfect
ground field would be sufficient.)
Generalized Reynolds ideals have originally been defined for symmetric
algebras (see \cite{BHZ} for an extension to arbitrary
finite-dimensional algebras).
Any finite-dimensional symmetric $F$-algebra $A$ has
an associative, symmetric, non-degenerate $F$-bilinear form
$\langle -,-\rangle : A
\times A \rightarrow F$. With respect to this form we have
for any subspace $M$ of $A$
the orthogonal space $M^{\bot}$.
Moreover, let $K(A)$ be the commutator subspace, i.e. the
$F$-subspace
of $A$ generated by all commutators $[a,b]:=a b - b a$,
where $a, b \in A$. For any $n \geq 0$ set
$$T_n (A) = \left\{ x \in A \mid x^{p^n} \in K(A)\right\}.$$
Then, by \cite{Ku1}, for any $n\ge 0$, the orthogonal
space $T_n (A)^{\bot}$ is an ideal of the center
$Z(A)$ of $A$. These are called {\it generalized Reynolds ideals}.
They form a descending sequence
$$Z(A) =  K(A)^{\perp} = T_0(A)^{\perp} \supseteq T_1(A)^{\perp}
\supseteq T_2(A)^{\perp}
\supseteq \ldots \supseteq T_n(A)^{\perp} \supseteq \ldots$$

In \cite{HHKM} it has been shown that the sequence of generalized
Reynolds ideals is invariant under Morita equivalences. More
generally, the following theorem has been proven recently by
the second author.

\begin{prop}[\cite{Z}, Theorem 1]
\label{prop:zimmermann}
Let $A$ and $B$ be finite-dimensional symmetric algebras over a perfect
field $F$ of positive characteristic $p$.
If $A$ and $B$ are derived equivalent, then
there is an isomorphism $\varphi : Z(A) \rightarrow Z(B)$ between
the centers of $A$ and $B$ such that $\varphi(T_n (A)^{\bot}) = T_n
(B)^{\bot}$ for all positive integers $n$.
\end{prop}

We note that in the proof of \cite[Theorem 1]{Z} the
fact that $F$ is algebraically closed is never used.
The assumption on the field $F$ to be perfect is sufficient.
Hence the sequence of generalized Reynolds ideals gives a new
derived invariant for symmetric algebras over perfect
fields of positive characteristic.

The aim of the present note is to show how these new derived
invariants can be applied to some subtle questions in the
derived equivalence classifications of algebras of dihedral
and semidihedral type.

\section{A symmetric bilinear form}

Symmetric algebras are equipped with an associative, non-degenerate
symmetric bilinear form.
For actual computations with generalized Reynolds ideals one needs
to know such a symmetrizing form explicitly. We should stress
that the series of generalized Reynolds
ideals is independent of the choice of symmetrizing form. Indeed,
a symmetrizing form is equivalent to an identification of 
$A$ with its dual as $A$-$A$-bimodules. Hence, two symmetrizing forms 
differ by an automorphism of $A$ as an $A$-$A$-bimodule, i.e., 
by a central unit of $A$. Computing the 
Reynolds ideals with respect to
another symmetrizing form therefore 
just means multiplying them by a central unit;
in particular, this  
leaves them invariant, since Reynolds ideals are ideals of the centre. 
%
The algebras in our paper are all basic symmetric
algebras, defined by a quiver with relations $A=FQ/I$. There is the
following standard construction, which
provides a bilinear form very suitable for
actual calculations.
As usual, $\soc(A)$
denotes the socle of the algebra $A$. Recall that an algebra is
called weakly symmetric if for each projective indecomposable
module the top and the socle are isomorphic.

\begin{prop}
\label{prop:form}
Let $A=FQ/I$ be a weakly symmetric algebra given by the quiver $Q$ and
ideal of relations $I$, and fix
an $F$-basis ${\mathcal B}$ of $A$ consisting of pairwise
distinct non-zero paths of the quiver $Q$. Assume that
${\mathcal B}$ contains a basis of $\soc(A)$.
Then the following statements hold:
\begin{enumerate}
\item[{(1)}] Define an $F$-linear mapping $\psi$ on the basis elements by
$$
\psi(b)=\left\{
\begin{array}{ll} 1 & \mbox{if $b\in\soc(A)$} \\
                  0 & \mbox{otherwise}
\end{array} \right.
$$
for $b\in {\mathcal B}$.
Then an associative non-degenerate $F$-bilinear
form $\langle-,-\rangle$ for $A$ is given by
$\langle x,y\rangle := \psi(xy).$
\item[{(2)}] If $A$ is symmetric, then for any
$n\ge 0$, the socle $\soc(A)$ is contained in
the generalized Reynolds ideal $T_n(A)^{\perp}$.
\end{enumerate}
\end{prop}

\begin{proof} (1)
By definition, since $A$ is an associative algebra,
$\psi$ is associative on basis elements,
hence is associative on all of $A$.

We observe now that $\psi(xe)=\psi(ex)$ for all $x\in A$ and all
primitive idempotents $e\in A$. Indeed, since $\psi$
is linear, we need to show this only on the elements
in $\mathcal B$. Let $b\in\mathcal B$. If $b$ is a path not in the
socle of $A$, then $be$ and $eb$ are either zero or not contained
in the socle either,
and hence $0=\psi(b)=\psi(be)=\psi(eb)$.
Moreover, by assumption $A$ is  weakly symmetric.
If $b\in{\mathcal B}$ is in the socle of $A$, then $b=e_bb=be_b$ for exactly
one primitive idempotent $e_b$ and $e'b=be'=0$ for each primitive idempotent
$e'\neq e_b$. Therefore, $\psi(e'b)=\psi(be')=0$ and
$\psi(e_bb)=\psi(b)=\psi(be_b)$.

It remains to show that the map $(x,y)\mapsto \psi(xy)$ is non-degenerate.
Suppose we had $x\in A\setminus\{0\}$ so that $\psi(xy)=0$ for all $y\in A$.
In particular for each primitive idempotent $e_i$ of $A$ we get
$\psi(e_ixy)=\psi(xye_i)=0$
for all $y\in A$.
Hence we may suppose that
$x\in e_iA$
for some primitive idempotent $e_i\in A$.

Now, $xA$ is a right $A$-module. Choose a simple submodule $S$ of $xA$ and
$s\in S\setminus\{0\}$. Then, since $s\in S\leq xA$ there is a $y\in A$ so that $s=xy$.
Since $S\leq xA\leq A$, and since $S$ is simple, $s\in \soc(A)\setminus\{0\}$.
Moreover, since $x\in e_iA$, also $s=e_is$, i.e. $s$ is in the (1-dimensional)
socle of the projective indecomposable module $e_iA$. So, up to a scalar factor,
$s$ is a path contained in the basis $\mathcal{B}$ (recall that
by assumption $\mathcal{B}$ contains a basis of the socle).
This implies that
$$\psi(xy)=\psi(s)=\psi(e_is)\neq 0,$$
contradicting the choice of $x$, and hence proving non-degeneracy.

\smallskip

(2) By \cite{HHKM} we have for any symmetric algebra $A$ that
$$\bigcap_{n=0}^{\infty} T_n(A)^{\perp} = \soc(A) \cap Z(A).$$
Moreover, using the proof given in \cite{Ku2}, for a basic algebra
for which the endomorphism rings of all simple modules are
commutative, we always have $\rad(A)\supseteq K(A)$ and hence,
taking orthogonal spaces,
$\soc(A)\subseteq Z(A)$.
\end{proof}

\begin{remark}\rm
We should mention that the hypothesis on the algebra $A$ in
the above proposition is satisfied for the algebras of dihedral
and semidihedral type we deal with in
this paper. Moreover, these algebras are symmetric algebras,
and for all of them the above-described form
$\langle-,-\rangle$ is actually symmetric
which can be checked directly from the definitions of the
algebras given below.
Hence, we shall use the form given in Proposition \ref{prop:form}
throughout
as symmetrizing form for our computations of generalized Reynolds
ideals.

With a more subtle analysis one might be able to show that if
$A=FQ/I$ as in the proposition is assumed to be symmetric then
the form $\langle-,-\rangle$ is always symmetric. We do not
embark on this aspect here.
\end{remark}

\section{Algebras of dihedral type}
\label{sec-dihedral}

Following K. Erdmann \cite[sec. VI.2]{Erdmann}, an algebra $A$
(over an algebraically closed field) is said to be {\em of
dihedral type} if it satisfies the following conditions:
\begin{itemize}
\item[{(i)}] $A$ is symmetric and indecomposable.
\item[{(ii)}] The Cartan matrix of $A$ is non-singular.
\item[{(iii)}] The stable Auslander-Reiten quiver of $A$ consists of
the following components: 1-tubes, at most two 3-tubes,
and non-periodic components of tree class
$\mathbb{A}_{\infty}^{\infty}$ or $\tilde{\mathbb{A}}_{1,2}$.
\end{itemize}

K. Erdmann classified these algebras up to Morita equivalence.
A derived equivalence classification of algebras of dihedral type
has been given in \cite{H-derived}. Any algebra of dihedral type
with two simple modules is derived
equivalent to a basic algebra $A^{k,s}_c:=D(2\mathcal{B})^{k,s}(c)$
where $k, s\ge 1$ are integers and the scalar is $c=0$ or $c=1$.
These algebras are defined by the following quiver
\medskip


\begin{center}
 \includegraphics[scale=0.6]{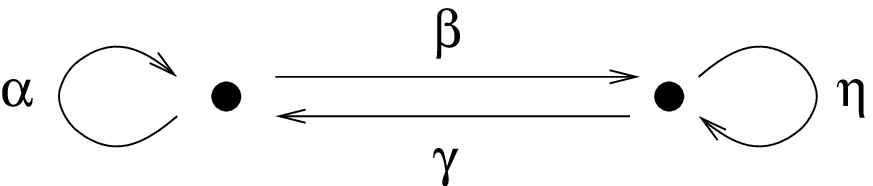}
\end{center}


subject to the relations
$$\beta\eta=0,~~ \eta\gamma=0,~~ \gamma\beta=0,~~
\alpha^2=c(\alpha\beta\gamma)^k,~~
(\alpha\beta\gamma)^k = (\beta\gamma\alpha)^k, ~~
\eta^s = (\gamma\alpha\beta)^k.
$$
Note that the case $s=1$ has to be interpreted so
that the loop $\eta$ doesn't exist in the quiver.

The algebras $A^{k,s}_0$ and $A^{k,s}_1$
are known to be isomorphic if the underlying field has characteristic
different from $2$ \cite[proof of VI.8.1]{Erdmann}.
So
we assume throughout this section that the underlying field
has characteristic 2.

For any $k,s\ge 1$ (and fixed $c$) the algebras $A^{k,s}_c$
and $A^{s,k}_c$ are derived equivalent \cite[lemma 3.2]{H-derived}.
So the derived equivalence classes are represented
by the algebras $A^{k,s}_c$ where $k\ge s\ge 1$ and $c\in\{0,1\}$.
Moreover, for different parameters $k',s'\ge 1$, i.e.
if $\{k,s\}\neq \{k',s'\}$, the algebra
$A^{k',s'}_d$ (where $d\in\{0,1\}$)
is not derived equivalent to $A^{k,s}_c$
\cite[lemma 3.3]{H-derived}.

Blocks of finite group algebras having dihedral defect group
of order $2^n$ and two simple modules are Morita equivalent
to algebras $D(2\mathcal{B})^{1,2^{n-2}}(c)$.

In this section we are going to study the sequence of
generalized Reynolds ideals
$$Z(A^{k,s}_c) \supseteq T_1(A^{k,s}_c)^{\perp}
\supseteq T_2(A^{k,s}_c)^{\perp}
\supseteq \ldots \supseteq T_r(A^{k,s}_c)^{\perp} \supseteq \soc(A^{k,s}_c)$$
of the center.
It is known by \cite{Z} that this sequence is invariant under
derived equivalences.

Our main result in this section is the following, partly
restating
Theorem \ref{thm-intro-dihedral} of the Introduction.

\begin{theorem} \label{thm-dihedral}
Let $k, s\ge 1$, and suppose that if $k=2$ then $s\ge 3$ is odd, and
if $s=2$ then $k\ge 3$ is odd.

Then the factor rings $Z(A^{k,s}_0)/T_1(A^{k,s}_0)^{\perp}$
and $Z(A^{k,s}_1)/T_1(A^{k,s}_1)^{\perp}$ are not isomorphic.

In particular, the algebras $A^{k,s}_0$ and $A^{k,s}_1$ are
not derived equivalent.
\end{theorem}

\begin{remark}
{\em

This result has already been obtained earlier
by M. Kauer \cite{Kauer},
\cite{KR}, using entirely different methods, as byproduct
of a rather sophisticated study of the class of so-called
graph algebras. (With this method, the cases of small
parameters excluded above can also be dealt with.)
However, our
new 'linear algebra' proof seems to be more elementary. Moreover,
our methods can successfully
be extended to algebras of semidihedral type, as
we shall see in the next section, in contrast to the methods
in \cite{Kauer}, \cite{KR}.
}
\end{remark}

Before embarking on the proof of Theorem \ref{thm-dihedral}
we need to collect some prerequisites, and thereby we also
set some notation.

\subsection{Bases for the algebras}
\label{subsec-basis}
We fix the integers $k,s\ge 1$.
We have to compute in detail with elements of the algebras
$A^{k,s}_c$ where $c=0$ or $c=1$.
Both algebras are of dimension $9k+s$ (cf. \cite{Erdmann}), the Cartan matrix is of
the form
$$\left( \begin{array}{cc}
4k & 2k \\ 2k & k+s
\end{array} \right).
$$
A basis of $A^{k,s}_c$ is given by
the union of the following bases of the subspaces
$e_iA^{k,s}_c e_j$, where $e_1$ and $e_2$ are the idempotents
corresponding to the trivial paths at the vertices of the quiver:
$$\mathcal{B}_{1,1}:=\{e_1, (\alpha\beta\gamma)^{i},
\alpha, (\alpha\beta\gamma)^{i}\alpha,
(\beta\gamma\alpha)^{i},\beta\gamma,(\beta\gamma\alpha)^{i}\beta\gamma,
(\alpha\beta\gamma)^k=(\beta\gamma\alpha)^k\,:\,1\le i\le k-1\}
$$
$$\mathcal{B}_{1,2}:=\{\beta,(\beta\gamma\alpha)^{i}\beta,
\alpha\beta,(\alpha\beta\gamma)^{i}\alpha\beta\,:\,1\le i\le k-1\}
$$
$$\mathcal{B}_{2,1}:=\{\gamma,
(\gamma\alpha\beta)^{i}\gamma,
\gamma\alpha,
(\gamma\alpha\beta)^{i}\gamma\alpha\,:\,1\le i\le k-1\}
$$
$$\mathcal{B}_{2,2}:=\{e_2, (\gamma\alpha\beta)^{i},
\eta, \eta^2,\ldots, \eta^{s-1}, \eta^s
\,:\,1\le i\le k-1\}.
$$
Note that this basis
$\mathcal{B}_{1,1}\cup \mathcal{B}_{1,2}\cup
\mathcal{B}_{2,1}\cup \mathcal{B}_{2,2}$
is independent of the scalar $c$.

\subsection{The centers}
\label{subsec-center}
The center of $A^{k,s}_c$ has dimension $k+s+2$ (cf. \cite{Erdmann}), a basis of
the center $Z(A^{k,s}_c)$ is given by
$$\mathcal{Z}:=
\{1,(\alpha\beta\gamma)^{i}+(\beta\gamma\alpha)^{i}+
(\gamma\alpha\beta)^{i}, (\beta\gamma\alpha)^{k-1}\beta\gamma,
(\alpha\beta\gamma)^k,
\eta^j\,:\,1\le i\le k-1,\, 1\le j\le s\}.
$$
Note that this basis is also independent of the scalar $c$.

\subsection{The commutator spaces}
\label{subsec-comm}
The algebras $A^{k,s}_c$ are symmetric, so
the commutator space $K(A^{k,s}_c)$ has dimension
$$\dim K(A^{k,s}_c) = \dim A^{k,s}_c - \dim Z(A^{k,s}_c) =
9k+s - (k+s+2) = 8k-2.$$
Indeed, the center of an algebra is the degree $0$ Hochschild cohomology of the algebra,
the quotient space of the algebra modulo the commutators is the degree $0$ Hochschild homology
of the algebra, and the $k$-linear dual of the Hochschild homology of an algebra is isomorphic
to the Hochschild cohomology of the algebra with values in the space of linear forms
of the the algebra
(cf. \cite[Chapter 1, Exercice 1.5.3, Corollary 1.1.8 and
Section 1.5.2]{Loday}).
A basis
of $K(A^{k,s}_c)$ is given by the union
$$\mathcal{K}:=\mathcal{B}_{1,2}\cup \mathcal{B}_{2,1}\cup
\mathcal{K}_{1,1}\cup \mathcal{K}_{2,2}$$
where $\mathcal{B}_{1,2}$ and $\mathcal{B}_{2,1}$
have been defined above and where
$$\mathcal{K}_{1,1}:=\{\beta\gamma,
(\alpha\beta\gamma)^{i}+(\beta\gamma\alpha)^{i},
(\alpha\beta\gamma)^{i}\alpha,
(\beta\gamma\alpha)^{i}
\beta\gamma\,:\, 1\le i\le k-1\}
$$
and
$$\mathcal{K}_{2,2}:=\{(\alpha\beta\gamma)^i+(\gamma\alpha\beta)^i\,:\,1\le i\le k\}.
$$

\subsection{The spaces $T_1$}
\label{subsec-T1}
We now consider the spaces
$$T_1(A^{k,s}_c):=\{ x\in A^{k,s}_c\,:\,x^2\in K(A^{k,s}_c)\}.$$
Note that the commutator space is always contained in
$T_1$ \cite[eq.\,(16)]{Ku2}.
Recall that a basis for $K(A^{k,s}_c)$ was given in
Section \ref{subsec-comm}.
The codimension of the commutator space inside the entire algebra
is
$$\dim A^{k,s}_c/K(A^{k,s}_c) = \dim Z(A^{k,s})= 9k+s - (8k-2) = k+s+2.$$

A basis of $A^{k,s}_c/K(A^{k,s}_c)$ is given by the cosets of the
following paths
\begin{equation} \label{basisA/KA}
\{e_1,e_2,\alpha,
\alpha\beta\gamma,\ldots,(\alpha\beta\gamma)^{k-1},
\eta,\ldots,\eta^{s-1},\eta^s\}.
\end{equation}
From this, we determine bases of the spaces $T_1(A^{k,s}_c)$.
It turns out that they depend on the parity of $k$ and $s$
(and on the scalar $c$). Recall that we denoted the above
basis of the commutator space by $\mathcal{K}$. By
$\lfloor.\rfloor$ and $\lceil.\rceil$ we denote the usual floor
and ceiling functions, respectively.

\begin{lemma}
\label{lem-T-dihedral}
A basis of $T_1(A^{k,s}_c)$ is given by the union
$$\mathcal{T}:= \mathcal{K}\cup
\{(\alpha\beta\gamma)^{\lceil \frac{k+1}{2}\rceil},\ldots,
(\alpha\beta\gamma)^{k-1},\eta^{\lceil \frac{s+1}{2}\rceil},\ldots,
\eta^s\}\,\cup\,\mathcal{N}
$$
where the set $\mathcal{N}$ is equal to
$$\left\{
\begin{array}{ll}
\{\alpha\} & \mbox{{\em if $c=0$ and $k$ or $s$ odd}} \\
\{\alpha, (\alpha\beta\gamma)^{\lfloor k/2\rfloor}+
     \eta^{\lfloor s/2\rfloor} \}
            & \mbox{{\em if $c=0$ and $k,s$ even}} \\
\emptyset  & \mbox{{\em if $c=1$ and $k,s$ odd}} \\
\{\alpha+\eta^{\lfloor s/2\rfloor}\} &
          \mbox{{\em if $c=1$ and $k$ odd, $s$ even}} \\
\{\alpha+(\alpha\beta\gamma)^{\lfloor k/2\rfloor} \} &
            \mbox{{\em if $c=1$ and $k$ even, $s$ odd}} \\
\{\alpha+(\alpha\beta\gamma)^{\lfloor k/2\rfloor},
(\alpha\beta\gamma)^{\lfloor k/2\rfloor}+\eta^{\lfloor s/2\rfloor}\}  &
          \mbox{{\em if $c=1$ and $k,s$ even}} \\
\end{array}
\right.
$$
\end{lemma}

\proof As mentioned above, the commutator space is always
contained in $T_1(A^{k,s}_c)$ \cite[eq.\,(16)]{Ku2}. So it
remains to deal with the elements outside the commutator,
and we use the basis of $A^{k,s}_c/K(A^{k,s}_c)$ given above in
(\ref{basisA/KA}). So we consider a linear combination
$$\lambda := a_0\alpha + a_1(\alpha\beta\gamma)+\ldots
+a_{k-1}(\alpha\beta\gamma)^{k-1} + b_1\eta +
\ldots + b_{s-1}\eta^{s-1} + b_s\eta ^s,$$
where $a_i,b_j\in F$, and the question is, when
is $\lambda^2\in K(A^{k,s}_c)$?
(Note that for this question the idempotents occurring
in the basis (\ref{basisA/KA}) can be
disregarded.) Since we are working in characteristic 2, we get
$$\lambda^2 = a_0^2\alpha^2 + \ldots + a_{\lfloor k/2\rfloor}^2
\left((\alpha\beta\gamma)^{\lfloor k/2\rfloor}\right)^2 + b_1^2\eta^2 +
\ldots + b_{\lfloor s/2\rfloor}^2 \left(\eta^{\lfloor s/2\rfloor}\right)^2 \;\;(\mbox{mod }K(A_c^{k,s})).
$$
Thus we can deduce that $\lambda^2\in K(A^{k,s}_c)$ if and only if
the following conditions are satisfied:
\begin{itemize}
\item[{(i)}] $a_1=\ldots = a_{\lfloor k/2\rfloor -1} = 0$ and
$b_1=\ldots = b_{\lfloor s/2\rfloor -1} = 0$
\item[{(ii)}]
$$\left\{ \begin{array}{ll}
a_{\lfloor k/2\rfloor} = 0 = b_{\lfloor s/2\rfloor} &
     \mbox{if $c=0$ and $k$ or $s$ odd} \\
 a_{\lfloor k/2\rfloor} = b_{\lfloor s/2\rfloor}  &
      \mbox{if $c=0$ and $k,s$ even} \\    
a_0=0,~ a_{\lfloor k/2\rfloor} = 0 = b_{\lfloor s/2\rfloor} &
     \mbox{if $c=1$ and $k,s$ odd} \\
a_0= b_{\lfloor s/2\rfloor}, ~ a_{\lfloor k/2\rfloor} = 0 &
          \mbox{if $c=1$ and $k$ odd, $s$ even} \\
a_0= a_{\lfloor k/2\rfloor}, ~ b_{\lfloor s/2\rfloor} = 0 &
      \mbox{if $c=1$ and $k$ even, $s$ odd} \\
 a_0 + a_{\lfloor k/2\rfloor} + b_{\lfloor s/2\rfloor} = 0  &
     \mbox{if $c=1$ and $k,s$ even} \\
\end{array}\right.
$$
\end{itemize}
\smallskip

These conditions directly translate into the statement
on the basis elements in the set $\mathcal{N}$, thus
proving the lemma.
\qed

\subsection{Proof of Theorem \ref{thm-dihedral}}
\label{subsec-proof-dihedral}
We are now in the position to prove Theorem \ref{thm-dihedral},
the main result of this section.
To this end, we have to
distinguish cases according to the parity of $k$ and $s$.
In each case we have to show that the algebras
$A^{k,s}_0$ and $A^{k,s}_1$ are not derived equivalent.

\subsubsection{Case $k,s$ odd}
By Lemma \ref{lem-T-dihedral}, the spaces $T_1(A^{k,s}_0)$ and
$T_1(A^{k,s}_1)$ have different dimensions. But these dimensions
are invariant under derived equivalences \cite[Theorem 1]{Z},
the dimension of the center being invariant and the bilinear form
being non-degenerate.
Hence, the algebras $A^{k,s}_0$ and $A^{k,s}_1$ are not derived
equivalent.

\subsubsection{Case $k$ odd, $s$ even}
\label{subsubsec-oddeven}

We first determine bases of the ideals $T_1(A^{k,s}_c)^{\perp}$.
Recall that these are ideals of the center $Z(A^{k,s}_c)$.
We are going to work with the bases $\mathcal{Z}$
of the center given in \ref{subsec-center}.
A straightforward computation yields that a basis for the orthogonal
space $T_1(A^{k,s}_c)^{\perp}$ is given by
$$\mathcal{T}^{\perp}:= \mathcal{N}' \cup
\{(\alpha\beta\gamma)^{i}+(\beta\gamma\alpha)^i+
(\gamma\alpha\beta)^i, (\alpha\beta\gamma)^k, \eta^j\,:\,
\lceil k/2\rceil\le i\le k-1,~s/2\le j\le s \}
$$
where
$$\mathcal{N}':=
\left\{ \begin{array}{ll}
\{\eta^{s/2}\} & \mbox{if $c=0$} \\
\{\eta^{s/2}+(\beta\gamma\alpha)^{k-1}\beta\gamma\} & \mbox{if $c=1$}
\end{array}
\right.
$$
We set $Z_c:=Z(A^{k,s}_c)$ for abbreviation and consider
the factor rings $\overline{Z}_c:=Z_c/T_1(A^{k,s}_c)^{\perp}$.

A basis of these factor rings can be given (independently
of $c$) by the cosets of the following central elements
$$\{ 1, (\alpha\beta\gamma)^i+(\beta\gamma\alpha)^i+
(\gamma\alpha\beta)^i, (\beta\gamma\alpha)^{k-1}\beta\gamma,
\eta^j\,:\, 1\le i\le \lceil k/2\rceil -1,~1\le j\le s/2-1\}.
$$
In order to show that these factor rings are not isomorphic,
we consider their Jacobson radicals
$\overline{J}_c:=\rad(\overline{Z}_c)$.
Clearly, a basis for $\overline{J}_c$ is obtained from the above
basis of $\overline{Z}_c$ by removing the unit element $1$.

The crucial observation now is that for $c=1$, we have that
$\eta^{s/2}=(\beta\gamma\alpha)^{k-1}\beta\gamma$.

If $s>2$ this implies that $(\beta\gamma\alpha)^{k-1}\beta\gamma$ is
contained in the square of the radical.
So the space $\overline{J}_1/\overline{J}^2_1$ has
dimension 2, spanned by the cosets of $\eta$ and $\alpha\beta\gamma+
\beta\gamma\alpha+\gamma\alpha\beta$.

On the other hand, if $c=0$ then $\overline{J}_0/\overline{J}^2_0$ has
dimension 3, spanned by the cosets of $\eta$, $\alpha\beta\gamma+
\beta\gamma\alpha+\gamma\alpha\beta$ and
$(\beta\gamma\alpha)^{k-1}\beta\gamma$.

Hence, if $s>2$, the factor rings $\overline{Z}_0$ and $\overline{Z}_1$
are not isomorphic. In particular, $A^{k,s}_0$
and $A^{k,s}_1$ are not derived equivalent.

By assumption we have that $s>2$ or $k>2$. The case $s=2$ and $k>2$
follows from the above argumentation using the fact
that the algebras $A^{k,s}_c$ and $A^{s,k}_c$ are
derived equivalent \cite[lem. 3.2]{H-derived}.

\subsubsection{Case $k$ even, $s$ odd}
This case follows from Subsection \ref{subsubsec-oddeven}
once we use that, for given $c$, the algebra $A^{k,s}_c$ is derived
equivalent to $A^{s,k}_c$ \cite[lem. 3.2]{H-derived}.

\subsubsection{Case $k,s$ even}
\label{subsubsec-even}

We first determine bases of the ideals $T_1(A^{k,s}_c)^{\perp}$.
Again, a direct calculation yields that a basis
for the orthogonal
space $T_1(A^{k,s}_c)^{\perp}$ is given by
$$\mathcal{T}^{\perp}:= \mathcal{N}' \cup
\{(\alpha\beta\gamma)^{i}+(\beta\gamma\alpha)^i+
(\gamma\alpha\beta)^i, (\alpha\beta\gamma)^k, \eta^j\,:\,
k/2+1\le i\le k-1,~s/2+1\le j\le s \}
$$
where
\begin{equation} \label{eq-Nperp}
\mathcal{N}':=
\left\{ \begin{array}{ll}
\{\eta^{s/2}+(\alpha\beta\gamma)^{k/2}+(\beta\gamma\alpha)^{k/2}+
(\gamma\alpha\beta)^{k/2}\} & \mbox{if $c=0$} \\
\{\eta^{s/2} +(\alpha\beta\gamma)^{k/2}+(\beta\gamma\alpha)^{k/2}+
(\gamma\alpha\beta)^{k/2} +
(\beta\gamma\alpha)^{k-1}\beta\gamma\} & \mbox{if $c=1$}
\end{array}
\right.
\end{equation}
As in Subsection \ref{subsubsec-oddeven},
we consider
the factor rings $\overline{Z}_c:=Z_c/T_1(A^{k,s}_c)^{\perp}$,
where $Z_c:=Z(A^{k,s}_c)$.
A basis of $\overline{Z}_c$ is given
by the cosets of the following central elements
$$\{ 1, (\alpha\beta\gamma)^i+(\beta\gamma\alpha)^i+
(\gamma\alpha\beta)^i, \eta^j, (\beta\gamma\alpha)^{k-1}\beta\gamma,
\eta^{s/2}\,:\, 1\le i\le k/2 -1,~1\le j\le s/2-1\}.
$$
Note that this basis is independent of the scalar $c$.

In order to show that these factor rings are not isomorphic,
we consider their Jacobson radicals
$\overline{J}_c:=\rad(\overline{Z}_c)$.
Clearly, a basis for $\overline{J}_c$ is obtained from the above
basis of $\overline{Z}_c$ by removing the unit element $1$.

The crucial observation now is that for $c=1$, it follows
from (\ref{eq-Nperp}) that in
$\overline{Z}_c$ we have
\begin{equation} \label{eq-Z1}
(\beta\gamma\alpha)^{k-1}\beta\gamma =
\eta^{s/2} + (\alpha\beta\gamma)^{k/2}+(\beta\gamma\alpha)^{k/2}+
(\gamma\alpha\beta)^{k/2}.
\end{equation}
On the other hand, if $c=0$, there is no relation whatsoever
in $\overline{Z}_0$ involving $(\beta\gamma\alpha)^{k-1}\beta\gamma$.

By assumption we have
that $k>2$ and $s>2$. Then equation (\ref{eq-Nperp})
implies that $(\beta\gamma\alpha)^{k-1}\beta\gamma
\in \overline{J}_1/\overline{J}^2_1$. Hence, for $c=1$
the space $\overline{J}_1/\overline{J}^2_1$ has dimension 2,
spanned by the cosets of $\eta$ and $\alpha\beta\gamma+
\beta\gamma\alpha+\gamma\alpha\beta$.
On the other hand, for $c=0$ the space
$\overline{J}_0/\overline{J}^2_0$ has dimension 3,
spanned by the cosets of $\eta$, $\alpha\beta\gamma+
\beta\gamma\alpha+\gamma\alpha\beta$ and
$(\beta\gamma\alpha)^{k-1}\beta\gamma$.

Hence the factor rings $\overline{Z}_0$ and
$\overline{Z}_1$ are not isomorphic. In particular, the algebras
$A^{k,s}_0$ and $A^{k,s}_1$ are not derived equivalent.

\section{Algebras of semidihedral type}
\label{sec-semidihedral}

Algebras of semidihedral type have been defined by Erdmann.
An algebra $A$
(over an algebraically closed field) is said to be {\em of
semidihedral type} if it satisfies the following conditions:
\begin{itemize}
\item[{(i)}] $A$ is symmetric and indecomposable.
\item[{(ii)}] The Cartan matrix of $A$ is non-singular.
\item[{(iii)}] The stable Auslander-Reiten quiver of $A$ has
the following components: tubes of rank at most 3, at most one 3-tube,
and non-periodic components isomorphic to
$\mathbb{Z}\mathbb{A}_{\infty}^{\infty}$ and
$\mathbb{Z}\mathbb{D}_{\infty}$.
\end{itemize}
Note that the original definition in \cite[VIII.1]{Erdmann}
contains the additional requirement that $A$ should be of tame
representation type. It has been shown by the
first author \cite[thm. 6.1]{H-derived} that tameness
already follows from the properties given in the above definition.

K. Erdmann gave a classification of algebras of semidihedral type
up to Morita equivalence.
A derived equivalence classification has been given
in \cite[sec.4]{H-derived}. It turns out that every
algebra of semidihedral type is derived equivalent to an algebra
in one of the two following families.

For any integers $k\ge 1$, $t\ge 2$ and a scalar $c\in\{0,1\}$
define the algebra $A^{k,t}_c=SD(2\mathcal{B})_1^{k,t}(c)$ by
the quiver
\smallskip


\begin{center}
 \includegraphics[scale=0.6]{quiver2b.eps}
\end{center}


subject to the relations
$$\gamma\beta=0, \eta\gamma=0, \beta\eta=0,
\alpha^2=(\beta\gamma\alpha)^{k-1}\beta\gamma+c(\alpha\beta\gamma)^k,
\eta^t=(\gamma\alpha\beta)^k,
(\alpha\beta\gamma)^k=(\beta\gamma\alpha)^k.$$

Secondly, for any $k\ge 1$, $t\ge 2$ such that $k+t\ge 4$ and
$c\in\{0,1\}$ we define the algebras
$B^{k,t}_c=SD(2\mathcal{B})_2^{k,t}(c)$ by the same quiver as
above, subject to the relations
$$\beta\eta=(\alpha\beta\gamma)^{k-1}\alpha\beta,
\gamma\beta=\eta^{t-1}, \eta\gamma=(\gamma\alpha\beta)^{k-1}\gamma\alpha,
\beta\eta^2=0, \eta^2\gamma=0, \alpha^2=c(\alpha\beta\gamma)^k.$$

We remark that every block of a finite group with semidihedral
defect group of order $2^n$ ($n\ge 4$)
and two simple modules is derived equivalent to
one of the following algebras:
$SD(2\mathcal{B})_1^{1,2^{n-2}}(c)$ or to
$SD(2\mathcal{B})_2^{2,2^{n-2}}(c)$ where
the scalar $c$ is either $0$ or $1$ and $p=2$. 

\subsection{The algebras $A^{k,t}_c$}

We first consider the algebras $A^{k,t}_c$ defined above.
The aim of this section is to prove Theorem \ref{thm1-intro-semidihedral},
distinguishing these algebras for different scalars up to derived
equivalence.

To this end, we are going to study the sequence of
generalized Reynolds ideals
$$\mathcal{Z}_c\,:\,Z(A^{k,t}_c) \supseteq T_1(A^{k,t}_c)^{\perp}
\supseteq T_2(A^{k,t}_c)^{\perp}
\supseteq \ldots \supseteq T_r(A^{k,t}_c)^{\perp} \supseteq \soc(A^{k,t}_c)$$
of the center.
\medskip

Let us compare the algebras $A^{k,t}_c$ of semidihedral type
defined above with the corresponding algebras
$A^{k,s}_c=D(2\mathcal{B})^{k,s}(c)$ of dihedral type
considered in Section \ref{sec-dihedral}. These algebras
are defined by the same quiver, and the only difference in
the relations is that now in the semidihedral case we have that
$\alpha^2=(\beta\gamma\alpha)^{k-1}\beta\gamma+c(\alpha\beta\gamma)^k$,
whereas we had
$\alpha^2=c(\alpha\beta\gamma)^k$ in the dihedral case.
Note that the new summand occurring,
$$(\beta\gamma\alpha)^{k-1}\beta\gamma=
[(\beta\gamma\alpha)^{k-1}\beta,\gamma]
$$
is a commutator in $A^{k,t}_c$ (using that $\gamma\beta=0$).
This actually means that the proof in the dihedral case given in
Subsections \ref{subsec-basis} - \ref{subsec-proof-dihedral}
carries over {\em verbatim} to the
algebras $A^{k,t}_c$ of semidihedral type.
We will therefore not repeat it.


\subsection{The algebras $B^{k,t}_c$}

We now consider the second family of algebras
$B^{k,s}_c=SD(2\mathcal{B})_2^{k,t}(c)$ where
$k\ge 1$, $t\ge 2$ such that $k+t\ge 4$ and $c\in\{0,1\}$.
The sequence of generalized Reynolds ideals takes the form
$$\mathcal{Z}_c\,:\,Z(B_c) \supseteq T_1(B_c)^{\perp} \supseteq
T_2(B_c)^{\perp}
\supseteq \ldots \supseteq T_r(B_c)^{\perp} \supseteq
\soc(B_c).$$

We can not distinguish the algebras completely, but we shall
prove the following partial result.

\begin{theorem} \label{thm-SD2-odd}
Suppose $k\ge 1$ and $t\ge 3$ are both odd. Then the spaces
$T_1(B^{k,t}_0)$ and $T_1(B^{k,t}_1)$ have different dimensions.

In particular, the algebras $B^{k,t}_0$ and $B^{k,t}_1$
are not derived equivalent.
\end{theorem}

\subsubsection{Bases for the algebras}
We fix integers $k\ge 1$ and $t\ge 2$ such that $k+t\ge 4$
(not necessarily both odd).
The algebras $B^{k,t}_c$ have dimension $9k+t$, the Cartan
matrix has the form (cf. \cite{Erdmann})
$$\left( \begin{array}{cc}
4k & 2k \\ 2k & k+t \end{array}
\right).
$$
A basis of the algebras, consisting of non-zero paths in the
quiver, is given by the union
$$\mathcal{B} := \mathcal{B}_{1,1}\cup \mathcal{B}_{1,2}
\cup \mathcal{B}_{2,1}\cup \mathcal{B}_{2,2},$$
where
$$\mathcal{B}_{1,1}:=\{e_1, (\alpha\beta\gamma)^{i},
\alpha, (\alpha\beta\gamma)^{i}\alpha,
(\beta\gamma\alpha)^{i},\beta\gamma,(\beta\gamma\alpha)^{i}\beta\gamma,
(\alpha\beta\gamma)^k=(\beta\gamma\alpha)^k=\beta\eta\gamma
\,:\,1\le i\le k-1\}
$$
$$\mathcal{B}_{1,2}:=\{\beta,(\beta\gamma\alpha)^{i}\beta,
\alpha\beta,(\alpha\beta\gamma)^{i}\alpha\beta\,:\,1\le i\le k-1\}
$$
$$\mathcal{B}_{2,1}:=\{\gamma,
(\gamma\alpha\beta)^{i}\gamma,
\gamma\alpha,
(\gamma\alpha\beta)^{i}\gamma\alpha\,:\,1\le i\le k-1\}
$$
$$\mathcal{B}_{2,2}:=\{e_2, (\gamma\alpha\beta)^{i},
\eta,\ldots, \eta^{t-2}, \eta^{t-1}=\gamma\beta,
\eta^t=(\gamma\alpha\beta)^k=\eta\gamma\beta
\,:\,1\le i\le k-1\}.
$$

\subsubsection{The centers}
\label{subsec-center-SD2}
The center of $B^{k,t}_c$ has dimension $k+t+2$ (cf. \cite{Erdmann}), a basis of
the center $Z(B^{k,t}_c)$ is given by
$$\mathcal{Z}:=
\{1,(\alpha\beta\gamma)^{i}+(\beta\gamma\alpha)^{i}+
(\gamma\alpha\beta)^{i}, (\beta\gamma\alpha)^{k-1}\beta\gamma,
(\alpha\beta\gamma)^k,
\eta^j, \eta+(\alpha\beta\gamma)^{k-1}\alpha
\,:\,1\le i\le k-1,\, 2\le j\le s\}.
$$
Note that this basis is also independent of the scalar $c$.

\subsubsection{The commutator spaces}
\label{subsec-comm-SD2}
The algebras $B^{k,t}_c$ are symmetric, so
the commutator space $K(B^{k,t}_c)$ has dimension
$$\dim K(B^{k,t}_c) = \dim B^{k,t}_c - \dim Z(B^{k,t}_c) =
9k+t - (k+t+2) = 8k-2.$$
A basis
of $K(B^{k,t}_c)$ is given by the union
$$\mathcal{K}:=\mathcal{B}_{1,2}\cup \mathcal{B}_{2,1}\cup
\mathcal{K}_{1,1}\cup \mathcal{K}_{2,2}$$
where $\mathcal{B}_{1,2}$ and $\mathcal{B}_{2,1}$
have been defined above and where
$$\mathcal{K}_{1,1}:=\{\beta\gamma+\gamma\beta,
(\alpha\beta\gamma)^{i}+(\beta\gamma\alpha)^{i},
(\alpha\beta\gamma)^{i}\alpha,
(\beta\gamma\alpha)^{i}
\beta\gamma\,:\, 1\le i\le k-1\}
$$
and
$$\mathcal{K}_{2,2}:=\{(\alpha\beta\gamma)^i+(\gamma\alpha\beta)^i,
\,:\,1\le i\le k\}.
$$

\subsubsection{The spaces $T_1$} We now consider the spaces
$$T_1(B^{k,t}_c):=\{ x\in B^{k,t}_c\,:\,x^2\in K(B^{k,t}_c)\}.$$
The commutator space is always contained in
$T_1$ \cite[eq.\,(16)]{Ku2}.
A basis for $K(B^{k,t}_c)$ was given in
Section \ref{subsec-comm-SD2}.
A basis of $B^{k,t}_c/K(B^{k,t}_c)$ is given by the cosets of the
following paths
\begin{equation} \label{basisA/KA-semi}
\{e_1,e_2,\alpha,
\alpha\beta\gamma,\ldots,(\alpha\beta\gamma)^{k-1},
\eta,\ldots,\eta^{t-1},\eta^t\}.
\end{equation}
We now turn to the
the spaces $T_1(B^{k,t}_c)$.
It turns out that they depend on the parity of $k$ and $s$
(and on the scalar $c$).

From now on, we assume that $k$ and $t$ are both odd.

Recall that we denoted the above
basis of the commutator space by $\mathcal{K}$.

\begin{lemma}
\label{lem-T-SD2}
Let $k\ge 1$ and $t\ge 3$ be both odd.
A basis of $T_1(B^{k,t}_c)$ is given by the union
$$\mathcal{T}:= \mathcal{K}\cup
\{(\alpha\beta\gamma)^{\frac{k+1}{2}},\ldots,
(\alpha\beta\gamma)^{k-1},\eta^{\frac{t+1}{2}},\ldots,
\eta^t\}\,\cup\,\mathcal{N}
$$
where the set $\mathcal{N}$ is equal to
$$\left\{
\begin{array}{ll}
\{\alpha\} & \mbox{{\em if $c=0$}} \\
\emptyset  & \mbox{{\em if $c=1$}} \\
\end{array}
\right.
$$
\end{lemma}

\proof
Since the commutator space is
contained in $T_1(B^{k,t}_c)$ \cite[eq.\,(16)]{Ku2}, it
remains to consider
the basis of $B^{k,t}_c/K(B^{k,t}_c)$ given in
(\ref{basisA/KA-semi}). So we consider a linear combination
$$\lambda := a_0\alpha + a_1(\alpha\beta\gamma)+\ldots
+a_{k-1}(\alpha\beta\gamma)^{k-1} + b_1\eta +
\ldots + b_{t-1}\eta^{t-1} + b_t\eta^t,$$
where $a_i,b_j\in F$, and we have to determine when
$\lambda^2\in K(B^{k,t}_c)$.
By assumption $k$ and $t$ are odd, so we get
$$\lambda^2 = a_0^2\alpha^2 + \ldots + a_{\frac{k-1}{2}}^2
(\alpha\beta\gamma)^{k-1} + b_1^2\eta^2 +
\ldots + b_{\frac{t-1}{2}}^2 \eta^{t-1}\;\;(\mbox{mod }K(B_c^{k,t}))
$$
(recall that we are working in characteristic 2).
A basis for the commutator space has been given
in \ref{subsec-comm-SD2}.
From this we can deduce that $\lambda^2\in K(B^{k,t}_c)$ if
and only if
the following conditions are satisfied:
\begin{itemize}
\item[{(i)}] $a_1=\ldots = a_{\frac{k-1}{2}} = 0$ and
$b_1=\ldots = b_{\frac{t-1}{2}} = 0$,
\item[{(ii)}] if $c=1$, also $a_0=0$.
\end{itemize}
\smallskip
From these conditions, the claim of the lemma follows directly.
\qed

\begin{remark}\rm
In case $k$ is even, in the square of $\lambda$ above a term $(\alpha\beta\gamma)^k$
appears and analogously to Lemma~\ref{lem-T-dihedral} it becomes impossible to
distinguish the parameters $c$ just by the dimensions of the generalized
Reynolds ideals. Similar phenomena appear for $t$ even.
\end{remark}

\subsubsection{Proof of Theorem \ref{thm-SD2-odd}}
From Lemma \ref{lem-T-SD2}, we deduce that the spaces
$T_1(B^{k,t}_0)$ and $T_1(B^{k,t}_1)$ have different
dimensions. But these dimensions are invariant
under derived equivalences, thus proving Theorem
\ref{thm-SD2-odd}.
\qed


\end{document}